\documentclass[12pt,reqno]{amsart}

\usepackage{amsmath,amssymb,amscd,graphicx}
\usepackage[all,2cell]{xy}
\usepackage{ifthen}
\usepackage{subsupscripts}
\usepackage[margin=1.25in]{geometry}	% See geometry.pdf to learn the layout options. There are lots.
\usepackage{url}
\usepackage{cite}
\usepackage{tikz}
\usepackage{hyperref}
\usepackage{enumerate}
\newboolean{sections}
\setboolean{sections}{true}
\UseComputerModernTips
\UseTwocells
\makeatletter
\def\th@plain{%
  \thm@notefont{}% same as heading font
  \itshape % body font
}
\def\th@definition{%
  \thm@notefont{}% same as heading font
  \normalfont % body font
}
\makeatother

\newcommand{\calG}{\mathcal{G}}

\newcommand{\id}{{\operatorname{id}}}  %identity
\newcommand{\pr}{{\operatorname{pr}}} %projection
\newcommand{\toto}{{~\rightrightarrows~}} %groupoid double arrow
\newcommand{\dashto}{{\;\dashrightarrow\;}} %Morita isomorphism/morphism

\makeatletter						% to be able to refer to a future numbered THM
\newtheorem*{rep@theorem}{\rep@title}
\newcommand{\newreptheorem}[2]{
\newenvironment{rep#1}[1]{
\def\rep@title{#2 ##1}
\begin{rep@theorem}}
{\end{rep@theorem}}}
\makeatother

\newcommand{\ifsection}[2]{\ifthenelse{\boolean{sections}}{#1}{#2}}

\theoremstyle{plain}
\ifsection{
    \newtheorem{theorem}{Theorem}[section]
    
	\numberwithin{equation}{section}
	\numberwithin{figure}{section}
	\usepackage{chngcntr}
	\counterwithin{table}{section}
}
{
    \newtheorem{theorem}{Theorem}
}

\newreptheorem{theorem}{Theorem}		%% future referenced numbered theorem
\newreptheorem{proposition}{Proposition}
\newreptheorem{corollary}{Corollary}
\newtheorem{proposition}[theorem]{Proposition}

\newtheorem{lemma}[theorem]{Lemma}

\theoremstyle{definition}
\newtheorem{definition}[theorem]{Definition}

\newtheorem{remark}[theorem]{Remark}

\newboolean{workmode}
\setboolean{workmode}{false}

\usepackage{tikz-cd}
\usepackage{mathtools}
\usepackage{amsfonts}
\newcommand{\dd}{\mathsf{d}} % differential in complex or Lie 2-algebra
\usepackage{amssymb,amsmath,amsthm,mathrsfs,xspace,multicol,stmaryrd}
\usepackage[mathscr]{eucal}
\usepackage{amscd}

\newcommand{\comment}[1]{}

\author{Dinamo Djounvouna}
\address{
Department of Mathematics, University of Manitoba, %\newline
Winnipeg, MB, Canada %R3T 2N2
}
\email{\href{mailto:djounvod@myumanitoba.ca}{djounvod@myumanitoba.ca}}

\author{Derek Krepski}
\address{
Department of Mathematics, University of Manitoba, %\newline
Winnipeg, MB, Canada %R3T 2N2
}
\email{\href{mailto:Derek.Krepski@umanitoba.ca}{Derek.Krepski@umanitoba.ca}}
\urladdr{\href{http://server.math.umanitoba.ca/~dkrepski/}{http://server.math.umanitoba.ca/\textasciitilde dkrepski/}}

\title[Bundle gerbe symmetries and Courant algebroids]{Infinitesimal symmetries of bundle gerbes and Courant algebroids}

\thanks{This work is partially supported by the  Natural Sciences and Engineering Research Council of Canada (RGPIN-2015-05833)}

\begin{document}

\begin{abstract}
Let $M$ be a smooth manifold and let $\chi\in \Omega^3(M)$ be closed differential form with integral periods. 
We show the Lie 2-algebra $\mathbb{L}(C_\chi)$ of sections of the $\chi$-twisted Courant algebroid $C_\chi$ on $M$ is quasi-isomorphic to the Lie 2-algebra of connection-preserving multiplicative vector fields on an $S^1$-bundle gerbe with connection (over $M$) whose 3-curvature is $\chi$.
% (a result originally due to Collier in 2011 using a stacks perspective). 
%The quasi-isomorphism obtained here is readily seen to be compatible with so-called gauge transformations $\chi \mapsto \chi + d\tau$ of the underlying 2-plectic manifold. As a consequence, we obtain a geometric argument, analogous to one of Miti and Zambon (2022) in the symplectic case, that establishes the embedding of the Lie 2-algebra of observables of $(M,\chi)$ into $\mathbb{L}(C_\chi)$ obtained by Rogers in 2010.
\end{abstract}

\maketitle
%\tableofcontents

\section{Introduction} \label{s:intro}

In  letters to A.\ Weinstein, P.\ \v{S}evera suggested that the infinitesimal symmetries of a \emph{Dixmier-Douady gerbe}, or $S^1$-gerbe, are closely related to  exact Courant algebroids  \cite{severa2017letters} (see also  \cite{bressler2005courant} where these ideas are further developed). 
In \cite{hitchin2003generalized,hitchin2006brackets}, Hitchin gives a construction of an exact Courant algebroid from the data of an $S^1$-gerbe over a manifold $M$, analogous to a construction of the Atiyah algebroid for principal $S^1$-bundles, which after a choice of splitting can be identified with the Courant algebroid $C_\chi=TM\oplus T^*M$ with \emph{$\chi$-twisted Courant bracket}, where $\chi\in \Omega^3(M)$ denotes the 3-curvature of the $S^1$-gerbe over $M$.

The relation to infinitesimal symmetries of $S^1$-gerbes was eventually made in Collier's PhD thesis \cite{collier-PhDThesis2012}. Viewing  $S^1$-gerbes $\calG$ as stacks (i.e., presheaves of groupoids), Collier identifies the infinitesimal symmetries of $S^1$-gerbes, showing they form a Lie 2-algebra\footnote{In this paper, Lie 2-algebras are 2-term $L_\infty$-algebras as in \cite{baez2004higher}.} $\mathcal{L}(\calG)$, and gives an alternate construction of an exact Courant algebroid $E_\calG$ (shown to be equivalent to that of Hitchin) from the data of the infinitesimal symmetries of  $\calG$.  As a consequence of \cite{roytenberg1998courant, roytenberg2002structure,roytenberg2007weak} (see also \cite{sheng2011semidirect} and \cite{rogers20132plectic}), the space of sections $\Gamma(E_\calG)$ can be given the structure of a Lie 2-algebra $\mathcal{L}(E_\calG)$, which Collier then shows to be quasi-isomorphic to the Lie 2-algebra $\mathcal{L}(\calG,\gamma)$ of infinitesimal symmetries of $\calG$ preserving a gerbe connection $\gamma$.

For $S^1$-gerbes $\calG$ with a connection $\gamma$ and curving $B$, Collier also considers the sub-Lie 2-algebra $\mathcal{L}(\calG;\gamma,B) \subset \mathcal{L}(\calG,\gamma)$ of infinitesimal symmetries preserving both $\gamma$ and $B$. 
 In \cite{fiorenza2014algebras}, Fiorenza, Rogers, and Schreiber give an interesting interpretation of this sub Lie 2-algebra and the Lie 2-algebra of all infinitesimal symmetries $\mathcal{L}(\calG)$. Specifically, they show that the natural sequence of  `forgetful' morphisms of Lie 2-algebras,
\begin{equation} \label{eq:seq}
\mathcal{L}(\calG;\gamma,B) \longrightarrow \mathcal{L}(\calG,\gamma) \longrightarrow \mathcal{L}(\calG),
\end{equation}
is equivalent  (via quasi-isomorphisms) to another sequence of Lie 2-algebra morphisms, which we recall next.

Consider the following three natural Lie 2-algebras  one may associate to a closed 3-form $\chi$ on a manifold $M$: the Lie 2-algebra of observables $\mathbb{L}(M,\chi)$, the Lie 2-algebra of sections of the $\chi$-twisted Courant algebroid $\mathbb{L}(C_\chi)$, and the \emph{skeletal} Lie 2-algebra $\mathbb{A}(M,\chi)$ associated to the representation of the Lie algebra of vector fields $\mathfrak{X}(M)$ on $C^\infty(M)$ with $C^\infty(M)$-valued 3-cocycle $\chi$ (called the \emph{Atiyah Lie 2-algebra} in \cite{fiorenza2014algebras}) --- see Section \ref{ss:lie2} for a brief review. 
There is a  sequence of morphisms of Lie 2-algebras,
\begin{equation} \label{eq:seq2}
\mathbb{L}(M,\chi) \longrightarrow \mathbb{L}(C_\chi) \longrightarrow \mathbb{A}(M,\chi),
\end{equation}
where the first morphism is an embedding defined  in  \cite{rogers20132plectic}, and the second is defined in \cite{fiorenza2014algebras}. In \emph{op.\ cit.}, the authors show that when $\chi$ is the 3-curvature of an $S^1$-gerbe $\calG$ with connection $\gamma$ and curving $B$, the sequence \eqref{eq:seq} is equivalent to the sequence \eqref{eq:seq2} via quasi-isomorphisms, where the middle quasi-isomorphism is that from \cite{collier-PhDThesis2012}.

There a several models for $S^1$-gerbes in the literature, and descriptions of infinitesimal symmetries for $S^1$-gerbes  thus depend on the choice of model. As stated above, for gerbes as presheaves of groupoids, infinitesimal symmetries are described in \cite{collier-PhDThesis2012}. In \emph{op.\ cit.}, Collier also describes infinitesimal symmetries for $S^1$-gerbes given in terms of \emph{\v{C}ech data}---i.e., Hitchin-Chatterjee gerbes \cite{hitchin2001lectures,chatterjee1998construction} or equivalently $S^1$-bundle gerbes \cite{murray1996bundle} where the underlying submersion is a covering by a disjoint union of open subsets. 
In \cite{fiorenza2014algebras}, Fiorenza, Rogers, and Schreiber,  in the more general \emph{higher structures} context,  describe the infinitesimal symmetries of $(n-1)$-bundle gerbes (or principal $U(1)$-$n$-bundles) viewed as \v{C}ech-Deligne cocycles. When $n=2$, these are  equivalent to Hithcin-Chaterjee gerbes, and the resulting Lie 2-algebras of infinitesimal symmetries  are essentially equivalent to those in \cite{collier-PhDThesis2012}.  

The perspective used in this paper is that from \cite{krepski2022multiplicative}, where infinitesimal symmetries of  $S^1$-bundle gerbes are modelled by \emph{multiplicative vector fields} on Lie groupoids.  In this viewpoint, an $S^1$-bundle gerbe $\calG$ over a manifold $M$ is an $S^1$-central extension of Lie groupoids $P \toto X$ of the submersion groupoid $X\times_M X \toto X$ associated to a surjective submersion $X\to M$.
 Multiplicative vector fields on a Lie groupoid form a category \cite{mackenzie1998classical}, and this category is a naturally a Lie 2-algebra  \cite{berwick2020lie}.  For an $S^1$-bundle gerbe $\calG = \{ P\toto X \}$, 
 we thus obtain a Lie 2-algebra of infinitesimal symmetries $\mathbb{X}(\calG)$ consisting of multiplicative vector fields on $P\toto X$.  When $\calG$ is equipped with a connection $\gamma$ and a curving $B$, multiplicative vector fields preserving $\gamma$ (resp. both $\gamma$ and $B$) in an appropriate `weak' sense form a Lie 2-algebra $\mathbb{X}(\calG,\gamma)$ (resp.\ $\mathbb{X}(\calG;\gamma,B)$)---see  Proposition \ref{p:lie2-connexpres} and \cite{krepski2022multiplicative}  for details. 
 For bundle gerbes  where $X$ is a disjoint union of open subsets of $M$, these Lie 2-algebras agree with those in \cite{collier-PhDThesis2012}.

We now describe the main contributions of this paper. 
In Theorem \ref{t:butterflyCourant}, stated for general $S^1$-bundle gerbes, we prove the expected analogue of Collier's quasi-isomorphism described above for Hitchin-Chaterjee gerbes. We work in Noohi's bicategory of Lie 2-algebras, with \emph{butterfly} morphisms \cite{noohi2013integrating}, and the desired quasi-isomorphism is realized as an invertible butterfly (see Section \ref{s:courant} for details). 

\begin{reptheorem} {\ref{t:butterflyCourant}'}
Let $\calG$ be an $S^1$-bundle gerbe with connection $\gamma$ over a manifold $M$. A choice of curving $B$ determines an invertible butterfly $\mathsf{F}: \mathbb{X}(\calG,\gamma) \dashto \mathbb{L}(C_\chi)$, where $\chi$ denotes the 3-curvature of the connection and curving.
\end{reptheorem}

Similarly, in Theorem \ref{t:atiyah} we show:

\begin{reptheorem}{\ref{t:atiyah}'}
Let $\calG$ be an $S^1$-bundle gerbe over a manifold $M$. A choice of connection $\gamma$ and curving $B$ determines an invertible butterfly $\mathsf{G}:\mathbb{X}(\calG) \dashto \mathbb{A}(M,\chi)$, where $\chi$ denotes the 3-curvature of the connection and curving.
\end{reptheorem}

In Propositions \ref{p:compatibility_gauge_transf} and \ref{p:atiyah_compatibility_gauge_transf} we show that the quasi-isomorphisms given by the invertible butterflies above have the expected compatibility with \emph{gauge transformations}, $\chi \mapsto \chi + d\tau$, $\tau \in \Omega^2(M)$, that accordingly alter the curving and 3-curvature of the bundle gerbe.

For multiplicative vector fields on a bundle gerbe $\calG$ with connection $\gamma$ and curving $B$, the natural sequence of `forgetful' morphisms
\[
\mathbb{X}(\calG;\gamma,B) \to \mathbb{X}(\calG, \gamma) \to \mathbb{X}(\calG)
\]
analogous to sequence \eqref{eq:seq} is shown here to be equivalent to the sequence of morphisms \eqref{eq:seq2}.  Indeed, the quasi-isomorphisms are supplied by \cite[Theorem 5.1]{krepski2022multiplicative}, Theorem \ref{t:butterflyCourant} , and Theorem \ref{t:atiyah}, while the desired 2-commutative diagrams follow from Propositions \ref{p:CourantRogersKVCompat} and \ref{p:FRScompat}.

As an application, we present in Section \ref{s:moment-gauge} a geometric argument analogous to one appearing in \cite{miti2022observables} in the symplectic case, showing Rogers' embedding of Lie 2-algebras (the first morphism in \eqref{eq:seq2}) is compatible with gauge transformations $\chi \mapsto \chi+d\tau$ after pulling back to a finite dimensional Lie algebra along a homotopy moment map.

\bigskip

\subsection*{Organization of the paper}

We recall some preliminaries in Section 2 on Lie 2-algebras in 2-plectic geometry and symmetries of $S^1$-bundle gerbes. Section 3 contains the main results in this paper, namely Theorems \ref{t:butterflyCourant} and \ref{t:atiyah}, as well as the compatibility of those results with gauge transformations $\chi \mapsto \chi + d\tau$. Finally, in Section 4 we give a geometric discussion analogous to one in \cite{miti2022observables} on the behaviour of the  Lie 2-algebra of observables under gauge transofrmation, as an application of the results in Section 3.

\section{Preliminaries} \label{s:prelim}

In this section we recall some background on Lie 2-algebras appearing in 2-plectic geometry, and some preliminaries on $S^1$-bundle gerbes and their infinitesimal symmetries.  We shall assume the reader is familiar with Lie 2-algebras---that is 2-term $L_\infty$-algebras as in \cite{baez2004higher}.  We localize Lie 2-algebras at weak equivalences (quasi-isomorphisms) and work within Noohi's bicategory of Lie 2-algebras, with \emph{butterflies} as 1-morphisms. We refer to Noohi's paper \cite{noohi2013integrating} for details, or \cite[Section 4.1]{krepski2022multiplicative} for a brief review of Lie 2-algebras.

\subsection{Lie 2-algebras in 2-plectic geometry} \label{ss:lie2}
In this subsection, we briefly recall three Lie 2-algebras naturally associated to closed 3-forms on smooth manifolds: the Poisson-Lie 2-algebra of observables, the Lie 2-algebra of sections of an exact Courant algebroid (see \cite{rogers20132plectic}), and the \emph{Atiyah Lie 2-algebra} of \cite{fiorenza2014algebras}. 

We begin with the (pre)-2-plectic analog of the Poisson algebra of observables on a symplectic manifold.

\begin{definition} \label{d:poisson-lie}
Let $M$ be a manifold, and let $\chi \in \Omega^3(M)$ be closed.
The \emph{Poisson-Lie 2-algebra (of observables)} $\mathbb{L}(M,\chi)$ is the Lie 2-algebra with underlying 2-term complex %$L_1 \to L_0$ given by
\[
C^\infty(M) \to \{(x,\beta) \in \mathfrak{X}(M) \times \Omega^1(M) \, | \, \iota_x \chi = -d\beta \},
\]
with differential $\dd f = (0,df)$; the bracket is given by
\[
[({x}_1,\beta_1),({x}_2,\beta_2)]=([{x}_1,{x}_2],\iota_{{x}_2}\iota_{{x}_1}\chi)
\]
in degree 0 and zero otherwise; the Jacobiator 
is given by
\[
J({x}_1,\beta_1;{x}_2,\beta_2;{x}_3,\beta_3) = - \iota_{{x}_3}\iota_{{x}_2}\iota_{{x}_1} \chi.
\]
\end{definition}

Recall that a closed 3-form $\chi$ on a manifold $M$ gives rise to an exact Courant algebroid,
$C_\chi=TM \oplus T^*M$ with $\chi$-twisted Courant bracket  \cite{severa2001poisson}. As noted in \cite{rogers20132plectic} (see also \cite{sheng2011semidirect}), it follows from Rotyenberg and Weinstein \cite{roytenberg1998courant} that sections of $C_\chi$ form a Lie 2-algebra, which is reviewed in the following definition.

\begin{definition} \label{d:courant}
Let $M$ be a manifold, and let $\chi\in \Omega^3(M)$ be closed. The \emph{Courant Lie 2-algebra} $\mathbb{L}(C_\chi)$ is the Lie 2-algebra with underlying 2-term complex 
given by
\[
C^\infty(M) \to \Gamma(TM\oplus T^*M)
\]
with differential $\dd f = (0,df)$; the bracket is given by
\[
[(u,\alpha),(v,\beta)] = ([u,v],L_u\beta - L_v \alpha - \frac{1}{2} d (\iota_u \beta - \iota_v \alpha) - \iota_v \iota_u \chi )
\]
in degree 0,
while in mixed degrees, we have
\[
[(u,\alpha),f]=-[f,(u,\alpha)] = \frac{1}{2} \iota_u df \, .
\]
The Jacobiator is given by
\[
J(u_1,\alpha_1;u_2,\alpha_2;u_3,\alpha_3) = -\frac{1}{6} \left( \langle [(u_1,\alpha_1),(u_2,\alpha_2)],(u_3,\alpha_3) \rangle^+
+ \text{cyc. perm.} \right) \, .
\]
\end{definition}

The notation $\langle -,-\rangle^+$ in Definition \ref{d:courant} denotes the standard symmetric pairing on $\Gamma(TM \oplus T^*M)$, $\langle (u,\alpha),(v,\beta) \rangle^+ = \iota_u \beta + \iota_v\alpha$.

Finally, we recall the construction of the  \emph{Atiyah Lie 2-algebra} associated to a manifold $M$ equipped with a closed 3-form $\chi$---namely, the skeletal Lie 2-algebra (see \cite{baez2004higher}) associated to the Lie algebra representation of $\mathfrak{X}(M)$ on $C^\infty(M)$ and $C^\infty(M)$-valued 3-cocycle $\chi$. More explicitly, 

\begin{definition} \label{d:atiyah} \cite{fiorenza2014algebras}
Let $M$ be a manifold, and let $\chi \in \Omega^3(M)$ be closed. The \emph{Atiyah Lie 2-algebra} $\mathbb{A}(M,\chi)$ is the Lie 2-algebra  with underlying 2-term complex 
\[
C^\infty(M) \to \mathfrak{X}(M)
\]
with zero differential $\dd(f) =0$, and bracket given by  Lie bracket of vector fields in degree 0 and Lie derivative $[{x},f]=-[f,{x}]=L_{x} f$ in mixed degree. The Jacobiator is given by 
\[
J({x}_1,{x}_2,{x}_3) = - \iota_{{x}_3}\iota_{{x}_2}\iota_{{x}_1} \chi.
\]
\end{definition}

\subsection{Bundle gerbes and their infinitesimal symmetries} \label{ss:gerbes}

\subsubsection*{Bundle gerbes and connective structures}
\label{ss:gerbes-inf}
We begin with a brief review of $S^1$-bundle gerbes to establish the perspective and notation. See \cite{behrend2011differentiable} for further details.

Recall that an $S^1$-\emph{bundle gerbe} $\calG$ over a manifold $M$ is an $S^1$-central extension of the submersion groupoid $X\times_M X \toto X$, where $\pi: X\to M$ is a surjective submersion. In more detail, this consists of a morphism of Lie groupoids
\[
\xymatrix{
P \ar@<-.7ex>[d]\ar@<.7ex>[d] \ar[r] & X\times_M X \ar@<-.7ex>[d]\ar@<.7ex>[d] \\
X \ar@{=}[r] & X
}
\]
and a left $S^1$-action on $P$ making $P\to X\times_M X$ a principal $S^1$-bundle
such that the $S^1$-action on $P$ is compatible with the groupoid multiplication:
\[
(zp)\cdot (wq) = (zw)(p\cdot q)
\]
for all composable $p,q \in P$ and $z,w\in S^1$.

A \emph{connection} on an $S^1$-bundle gerbe $\calG$ over $M$ is a connection 1-form $\gamma \in \Omega^1(P)$ that is multiplicative (i.e., $m^*\gamma = \pr_1^*\gamma+ \pr_2^*\gamma$, where $m:P\times_X P \to P$ denotes the Lie groupoid multiplication on $P$ and $\pr_1, \pr_2$ denote the obvious projections). Equivalently, letting $\delta$ denote the simplicial differential on the simplicial manifold $P_\bullet$ associated to the Lie groupoid $P\toto X$, we see a connection 1-form $\gamma$ on $P$ defines a connection on $\calG$ whenever $\delta \gamma =0$.

Given a connection $\gamma$ on $\calG$, a \emph{curving} for $\gamma$ is a 2-form $B\in \Omega^2(X)$ such that $\delta B = d\gamma$. In this case, we say the pair $(\gamma, B)$ defines a \emph{connective structure} on $\calG$. The 3-curvature of the connective structure $(\gamma,B)$ is the 3-form $\chi \in \Omega^3(M)$ satisfying $\pi^*\chi = dB$.

\begin{remark} \label{r:gauge-curving}
Observe that for a fixed connection $\gamma$ on an $S^1$-bundle gerbe $\calG$ over $M$, the set of curvings for $\gamma$ is a $\Omega^2(M)$-torsor: indeed, any two curvings $B$, $B'$ satisfy $\delta(B-B') = 0$ and thus there exists a unique 2-form $\tau \in \Omega^2(M)$ with $\pi^*\tau=B-B'$. Moreover, if $\chi$ denotes the 3-curvature of the connective structure $(\gamma,B)$, then $\chi+d\tau$ is the 3-curvature of the connective structure $(\gamma, B+\pi^*\tau)$.
\end{remark}

\subsubsection*{Multiplicative vector fields on bundle gerbes}

In \cite{krepski2022multiplicative}, infinitesimal symmetries of $S^1$-bundle gerbes were modelled with multiplicative vector fields on Lie groupoids, which naturally come with the structure of a Lie 2-algebra (\emph{cf}.\ \cite{berwick2020lie}). In particular, \cite{krepski2022multiplicative} considers multiplicative vector fields on an $S^1$-gerbe $\calG$ over $M$ that preserve a connective structure $(\gamma, B)$ on $\calG$. In this work, we consider multiplicative vector fields on $\calG$ that preserve the connection $\gamma$ (but not necessarily a curving $B$).

First, we briefly review the (strict) Lie 2-algebra of multiplicative vector fields on a general Lie groupoid $\mathbf{G}=\{G_1\toto G_0\}$.
Recall from \cite{mackenzie1998classical} that a \emph{multiplicative vector field} on a Lie groupoid $\mathbf{G}$ is a functor $\mathbf{x}:\mathbf{G} \to T\mathbf{G}$ such that $\pi_{\mathbf{G}}\circ \mathbf{x} = \id_{\mathbf{G}}$, where $\pi_{\mathbf{G}}:T\mathbf{G} \to \mathbf{G}$ denotes the tangent bundle projection. Such a functor $\mathbf{x}$ therefore consists of a pair of vector fields $(\mathbf{x}_0,\mathbf{x}_1) \in \mathfrak{X}(G_0) \times \mathfrak{X}(G_1)$ that are compatible with units and the groupoid multiplications on $\mathbf{G}$ and $T\mathbf{G}$. Denote the multiplicative vector fields on $\mathbf{G}$ (viewed as pairs of vector fields $(\mathbf{x}_0,\mathbf{x}_1)$ as above) by $\mathbb{X}(\mathbf{G})_0$.

Let $A=\ker ds \big|_{G_0}$ denote the Lie algebroid of $\mathbf{G}$, with anchor $dt:A\to TG_0$. A section $a\in \Gamma(A)$ gives rise to a multiplicative vector field as follows. Let ${\mathbf{a}} = dt(a)$ and ${\bar{\mathbf{a}}} = \overrightarrow{a} + \overleftarrow{a}$, where
\[
\overrightarrow{a}(g)=dR_g(a(t(g))) \quad \text{and} \quad \overleftarrow{a}(g) = d(L_g \circ i)(a(s(g))).
\]
Here, $L_g$ and $R_g$ denote left and right multiplication, respectively, by $g\in G_1$, and $i: G_1\to G_1$ denotes inversion. It follows that $({\mathbf{a}},{\bar{\mathbf{a}}})$ is a multiplicative vector field \cite[Example 3.4]{mackenzie1998classical}, and we may
obtain a Lie 2-algebra $\mathbb{X}(\mathbf{G})$, with underlying 2-term complex
\[
\Gamma(A) \longrightarrow \mathbb{X}(\mathbf{G})_0
\]
and differential $\dd a =({\mathbf{a}},{\bar{\mathbf{a}}})$.
The bracket of elements in degree 0 is given on components:
$$
[({\mathbf{x}_0}, {\mathbf{x}_1}), ({\mathbf{x}_0'}, {\mathbf{x}_1'})] = ([{\mathbf{x}_0},{\mathbf{x}_0'}], [{\mathbf{x}_1},{\mathbf{x}_1'}]),
$$
while in mixed degrees, it is given by
$$
[({\mathbf{x}_0}, {\mathbf{x}_1}),a] =-[a,({\mathbf{x}_0}, {\mathbf{x}_1})]= [{\mathbf{x}_1},\overrightarrow{a}]\big|_{G_0}.
$$
Here, recall that for $a\in \Gamma(A)$, and $({\mathbf{x}_0},{\mathbf{x}_1})$ multiplicative, $[{\mathbf{x}_1},\overrightarrow{a}] \in \ker ds$ and is right-invariant \cite{mackenzie1998classical}, and hence its restriction to $G_0$ defines a section in $\Gamma(A)$.

\medskip

Let $\calG$ be an $S^1$-bundle gerbe on a manifold $M$, where $\mathbf{P}=\{P\toto X\}$ denotes the underlying $S^1$-central extension. Let $\gamma$ be a connection on $\calG$.
Recall from \cite[Remark 3.15]{krepski2022multiplicative} that a multiplicative vector field $(\mathbf{x},\mathbf{p}) \in \mathbb{X}(\mathbf{P})_0$ (weakly) \emph{preserves the connection} $\gamma$ if $L_\mathbf{p} \gamma = \delta \alpha$ for some $\alpha \in \Omega^1(X)$.

\begin{lemma} \label{l:morphs-connex-pres}
Let $(\calG,\gamma)$ be an $S^1$-bundle gerbe with connection on a manifold $M$. Let $a\in \Gamma(A_P)$, where $A_P\to X$ denotes the Lie algebroid of the underlying Lie groupoid $P \toto X$. Then for any curving $B$, $L_{\bar{\mathbf{a}}} \gamma = \delta(\iota_{\mathbf{a}} B - d\mathsf{v}_a)$, where $\mathsf{v}_a = \epsilon^* \iota_{\overrightarrow{a}} \gamma$ (here, $\epsilon$ denotes the unit map for $P \toto X\times_M X$). Moreover, $\iota_{\mathbf{a}} B - d\mathsf{v}_a$ is independent of the choice of curving $B$.
\end{lemma}
\begin{proof}
The first claim is checked in the proof of \cite[Proposition 3.16]{krepski2022multiplicative}. To see that $\iota_{\mathbf{a}} B - d\mathsf{v}_a$ is independent of the curving $B$, by Remark \ref{r:gauge-curving}, it suffices to check that $\iota_{\mathbf{a}} \pi^*\tau=0$ for $\tau\in \Omega^2(M)$. By definition of $P\toto X$, $\pi \circ t = \pi \circ s$; therefore, $d\pi(dt(a))=d\pi(ds(a))=0$. That is, $\mathbf{a} \sim_\pi 0$ and the claim follows.
\end{proof}

Connection preserving multiplicative vector fields form a Lie 2-algebra $\mathbb{X}(\calG,\gamma)$, defined in the following Proposition.

\begin{proposition} \label{p:lie2-connexpres}
Let $(\calG,\gamma)$ be an $S^1$-bundle gerbe with connection on a manifold $M$. Let $\mathbb{X}(\calG, \gamma)$ denote the 2-term complex 
\[
\Gamma(A_P) \longrightarrow \{ (\mathbf{x},\mathbf{p},\alpha) \in \mathbb{X}(\mathbf{P})_0 \times \Omega^1(X) \, \big| \, L_\mathbf{p} \gamma = \delta \alpha \}
\]
with differential given by $\dd a= (\mathbf{a}, \overline{\mathbf{a}}, \iota_{\mathbf{a}} B - d\mathsf{v}_a)$ (with $B$ any curving for the connection $\gamma$). Define a bracket on elements of degree 0 by,
\[
[(\mathbf{x},\mathbf{p},\alpha),(\mathbf{z},\mathbf{r},\beta)] =
([\mathbf{x},\mathbf{z}],[\mathbf{p},\mathbf{r}],L_\mathbf{x}\beta - L_\mathbf{z} \alpha),
\]
while for mixed degree elements, set
\[
[(\mathbf{x},\mathbf{p},\alpha),a] = -[a,(\mathbf{x},\mathbf{p},\alpha)] =[\mathbf{p},\overrightarrow{a}] \big|_X\, .
\]
Then $\mathbb{X}(\calG,\gamma)$ is a strict Lie 2-algebra.
\end{proposition}
\begin{proof}
The proof is the same as that of \cite[Proposition 4.8]{krepski2022multiplicative}, save for the verification of the condition,
\[
\dd [(\mathbf{x},\mathbf{p},\alpha),a] = [(\mathbf{x},\mathbf{p},\alpha),\dd a].
\]
To check this, choose a curving $B$ and observe first that $\delta(L_\mathbf{x} B - d\alpha) = 0$ and hence there exists $\beta \in \Omega^2(M)$ with $\pi^*\beta = L_\mathbf{x} B - d\alpha$. Therefore, $\iota_{\mathbf{a}}(L_\mathbf{x} B - d\alpha) = \iota_{\mathbf{a}} \pi^*\beta =0$ since $\mathbf{a}\sim_\pi 0$ as observed in the proof of Lemma \ref{l:morphs-connex-pres}. The verification in \emph{loc.\ cit}.\ is now easily adapted.
\end{proof}

\section{The Courant algebroid and infinitesimal symmetries of bundle gerbes} \label{s:courant}

Let $(\calG,\gamma)$ be an $S^1$-bundle gerbe with connection over $M$ with underlying central $S^1$-extension $P\toto X$, and suppose $B$ is a curving for $\gamma$ with resulting 3-curvature $\chi\in \Omega^3(M)$.  In  Section \ref{ss:main}, we establish the main results of the paper.  Theorem \ref{t:butterflyCourant} gives an invertible butterfly between the Lie 2-algebra $\mathbb{X}(\calG,\gamma)$ of connection-preserving multiplicative vector fields on $\calG$ and  the Courant Lie 2-algebra $\mathbb{L}(C_\chi)$. In Theorem \ref{t:atiyah}, we also give an invertible butterfly between multiplicative vector fields $\mathbb{X}(\mathbf{P})$ on $\calG$ and the Atiyah Lie 2-algebra $\mathbb{A}(M,\chi)$. In Section \ref{ss:courant-gauge} we show these invertible butterfies are compatible with gauge transformations $\chi \mapsto \chi+d\tau$, where $\tau \in \Omega^2(M)$.

\subsection{Sections of the Courant algebroid as infinitesimal symmetries of a bundle gerbe} \label{ss:main}

Let $\calG = P\toto X$ be an $S^1$-bundle gerbe over $M$ and let $\gamma$ be a connection on $\calG$ and choose a curving $B$. Denote the resulting 3-curvature by $\chi \in \Omega^3(M)$. Below we construct an invertible butterfly between sections of the Courant algebroid and multiplicative vector fields on $\calG$ preserving the connection.

Let $F=\left\{(\mathbf{x}, \mathbf{p},\alpha; g) \in \mathbb{X}(\mathbf{P},\gamma)_0 \times C^\infty(X) \, \big| \,
\delta g = \iota_{\mathbf{p}} \gamma
\right \}$, and define the structure maps in the diagram below as follows.
\begin{equation}
\label{eq:butterflyF}
\begin{gathered}
\xymatrix@R=1em{
\Gamma(A_P) \ar[dd] \ar[dr]^-{\kappa} & & C^\infty(M) \ar[dd] \ar[dl]_-\lambda \\
& F\ar[dl]^-\sigma \ar[dr]_-\rho & \\
\mathbb{X}(\mathbf{P},\gamma)_0 & & \Gamma(TM\oplus T^*M)
}
\end{gathered}
\end{equation}

Let $\sigma = \pr_1$ denote the obvious projection. To define $\rho$, first note that $\delta (\alpha - \iota_\mathbf{x} B - dg) =0$, hence there exists a unique 1-form $\varepsilon \in \Omega^1(M)$ satisfying $\pi^*\varepsilon = \alpha - \iota_\mathbf{x} B - dg $. Set $\rho(\mathbf{x}, \mathbf{p},\alpha; g) = ({x},-\varepsilon)$, where ${x}$ is the vector field on $M$ onto which $\mathbf{x}$ projects. Finally, let $\lambda(f) = (0,0,0;\pi^*f)$, and
$\kappa(a)=(\dd a;-\mathsf{v}_a)$.

\begin{theorem} \label{t:butterflyCourant}
Let $(\calG,\gamma)$ be an $S^1$-bundle gerbe over $M$ with connection $\gamma$ and suppose $B$ is a curving for $\gamma$ with resulting 3-curvature $\chi\in \Omega^3(M)$. Let $F$ and the indicated structure maps be as above, and define a bracket on $F$ by the formula
\[
[(\mathbf{x},\mathbf{p}, \alpha;g) , (\mathbf{z},\mathbf{r}, \beta;h)]
= ([(\mathbf{x},\mathbf{p}, \alpha) , (\mathbf{z},\mathbf{r}, \beta)] ,
\frac{1}{2}\left( \iota_\mathbf{x} (\beta+dh) - \iota_\mathbf{z} (\alpha+dg) \right) ).
\]
Then $F$ defines an invertible butterfly $\mathsf{F}:\mathbb{X}(\calG,\gamma) \dashto \mathbb{L}(C_\chi)$.
\end{theorem}
\begin{proof}
We note that the underlying vector space $F$ of the butterfly, together with the indicated structure maps, are almost identical to those appearing in \cite[Theorem 5.1]{krepski2022multiplicative}; therefore, the commutativity of the triangles in the diagram \eqref{eq:butterflyF} and the exactness of the diagonal sequences follows for the same reasons as in \emph{loc.\ cit.}
It remains to check the compatibility of the bracket with the various structure maps and the Jacobiator. These verifications are all routine computations using the Cartan calculus of differential forms.
\end{proof}

In \cite{rogers20132plectic}, Rogers exhibits an embedding of Lie 2-algebras $\mathsf{R}:\mathbb{L}(M,\chi) \hookrightarrow \mathbb{L}(C_\chi)$. In  \cite[Theorem 5.1]{krepski2022multiplicative}, the authors describe a \emph{prequantization butterfly}, an invertible butterfly $\mathsf{E}:\mathbb{L}(M,\chi) \dashto \mathbb{X}(\calG,\gamma, B)$, where $\mathbb{X}(\calG,\gamma,B)$ denotes the sub-Lie 2-algebra of $\mathbb{X}(\calG,\gamma)$ consisting of multiplicative vector fields preserving (both) the connection and curving of the bundle gerbe. Proposition \ref{p:CourantRogersKVCompat} below shows that the butterfly $\mathsf{F}$ from Theorem \ref{t:butterflyCourant} is compatible with Rogers' embedding and the prequantization butterfly $\mathsf{E}$.

\begin{proposition} \label{p:CourantRogersKVCompat}
Let $(\calG,\gamma)$ be an $S^1$-bundle gerbe over $M$ with connection $\gamma$ and suppose $B$ is a curving for $\gamma$ with resulting 3-curvature $\chi\in \Omega^3(M)$. Let $\mathsf{F}$ be as in Theorem \ref{t:butterflyCourant}.
Then the following diagram 2-commutes:
\[
\xymatrix{
\mathbb{L}(M,\chi) \ar^{\mathsf{R}}[r] \ar@{-->}_{\mathsf{E}}[d] & \mathbb{L}(C_\chi) \\
\mathbb{X}(\calG;B,\gamma) \ar[r] & \mathbb{X}(\calG,\gamma) \ar@{-->}_{\mathsf{F}}[u] \ultwocell\omit{}
}
\]
where the map $\mathsf{R}$ is Rogers' embedding, and $\mathsf{E}$ is the prequantization butterfly.
\end{proposition}
\begin{proof}
We show the butterflies $\mathsf{R}\circ \mathsf{E}^{-1}$ and $\mathsf{F} \circ \mathsf{j}$ are isomorphic, where $\mathsf{j}$ denotes the inclusion $\mathbb{X}(\calG, \gamma, B) \to \mathbb{X}(\calG,\gamma)$.

Since the inclusion $\mathsf{j}$ is a strict morphism, the underlying vector space for the butterfly $\mathsf{F} \circ \mathsf{j}$ is simply the restriction $F\big|_{\mathbb{X}(\mathbf{P},\gamma,B)_0}$ (see \cite[Section 5.1]{noohi2013integrating}), which coincides with the underlying vector space $E$ of the prequantization butterfly $\mathsf{E}$.

The underlying chain map for Rogers' embedding $\mathbb{L}(M,\chi) \to \mathbb{L}(C_\chi)$ in our notation is given by inclusion in degree 0 and the identity in degree 1. Let $L_1\to L_0$ denote the underlying 2-term complex of $\mathbb{L}(M,\chi)$, and $K_1\to K_0$ the underlying 2-term complex of $\mathbb{L}(C_\chi)$.
Therefore, the underlying vector space for the corresponding butterfly $\mathsf{R}$ is $K_1 \oplus L_0$; hence for $\mathsf{R}\circ \mathsf{E}^{-1}$ it is
$E \mathrel{\substack{{L_1}\\\oplus\\{L_0}}} (K_1 \oplus L_0) \cong (E \oplus K_1)/L_1$ (quotient by diagonal image of $L_1$). Since $L_1=K_1$, we also have a natural isomorphism $ E \cong (E \oplus K_1)/L_1$ (inclusion into first summand) with inverse obtained by choosing a representative with trivial second summand.

The chain homotopy $R:L_0 \otimes L_0 \to K_1$ is given by
\[
R(({x},\beta),({z},\varphi)) = -\tfrac{1}{2} \left( \iota_{x} \varphi - \iota_{z} \beta \right).
\]
Therefore the induced bracket on $(E \oplus K_1)/L_1$ is given by
\begin{align*}
[(\mathbf{x},\mathbf{p}, & \alpha, g; f), (\mathbf{z},\mathbf{r},\beta,h;k)] \\
& = ([\mathbf{x},\mathbf{z}], [\mathbf{p},\mathbf{r}], L_\mathbf{x} \beta - L_\mathbf{y} \alpha, \iota_\mathbf{x} \beta - \iota_\mathbf{z} \alpha + \iota_\mathbf{z} \iota_\mathbf{x} B ;\tfrac{1}{2}( \iota_{x} (dk + \varpi) - \iota_{z} (df + \varepsilon) )
\end{align*}
where $({x},-\varepsilon)$ and $({z},-\varpi)$ denote elements in $L_0$ defined by
\begin{equation} \label{eq:epsilons}
\pi^*\varepsilon = \alpha - \iota_\mathbf{x} B - dg \quad \text{and} \quad
\pi^*\varpi = \beta - \iota_\mathbf{z} B - dh.
\end{equation}
Under the identification $(E\oplus K_1)/L_1 \cong E$, this reads
\begin{align*}
[(\mathbf{x},\mathbf{p}, &\alpha, g),(\mathbf{z},\mathbf{r},\beta,h)] \\
& = ([\mathbf{x},\mathbf{z}], [\mathbf{p},\mathbf{r}], L_\mathbf{x} \beta - L_\mathbf{y} \alpha, \iota_\mathbf{x} \beta - \iota_\mathbf{z} \alpha + \iota_\mathbf{z} \iota_\mathbf{x} B - \tfrac{1}{2}\pi^*( \iota_{x} \varpi - \iota_{z} \varepsilon) )
\end{align*}
Using \eqref{eq:epsilons}, we see this bracket agrees with the bracket on $F\big|_{\mathbb{X}(\mathbf{P},\gamma,B)_0}$.
\end{proof}

The butterfly in Theorem \ref{t:butterflyCourant} may be readily adjusted to give a similar butterfly $\mathsf{G}:\mathbb{X}(\mathbf{P}) \dashto \mathbb{A}(M,\chi)$. Indeed, let $G= \{ (\mathbf{x},\mathbf{p},g) \in \mathbb{X}(\mathbf{P})_0 \times C^\infty(X) \, | \, \delta g = \iota_{\mathbf{p}}\gamma\}$, and define the structure maps in the diagram below in the obvious way analogous to those in diagram \eqref{eq:butterflyF}.
\begin{equation}
\label{eq:butterflyG}
\begin{gathered}
\xymatrix@R=1em{
\Gamma(A_P) \ar[dd] \ar[dr]^-{\kappa} & & C^\infty(M) \ar[dd] \ar[dl]_-\lambda \\
& G\ar[dl]^-\sigma \ar[dr]_-\rho & \\
\mathbb{X}(\mathbf{P})_0 & & \mathfrak{X}(M)
}
\end{gathered}
\end{equation}
With this butterfly, we obtain the following Theorem, which is entirely analogous to Theorem \ref{t:butterflyCourant}. We omit the proof, since it uses the same methods and ideas as that of Theorem \ref{t:butterflyCourant}.

\begin{theorem} \label{t:atiyah}
Let $\calG = P\toto X$ be an $S^1$-bundle gerbe over $M$ with connection $\gamma$ and curving $B$, with resulting 3-curvature $\chi$.  Let $G$ and the indicated structure maps be as above, and define a bracket on $G$ by the formula
\[
[(\mathbf{x},\mathbf{p}, g) , (\mathbf{z},\mathbf{r}, h)]
= ([(\mathbf{x},\mathbf{p}) , (\mathbf{z},\mathbf{r})] ,
L_\mathbf{x} h - L_\mathbf{z} g - \iota_\mathbf{z}  \iota_\mathbf{x}  B  ).
\]
Then $G$ defines an invertible butterfly $\mathsf{G}:\mathbb{X}(\mathbf{P}) \dashto \mathbb{A}(M,\chi)$.

\end{theorem}

In \cite{fiorenza2014algebras}, the authors give a morphism of Lie 2-algebras $\psi:\mathbb{L}(C_\chi) \to \mathbb{A}(M,\chi)$. The following Proposition, analogous to Proposition \ref{p:CourantRogersKVCompat}, shows the butterflies of Theorems \ref{t:butterflyCourant} and \ref{t:atiyah} are compatible with $\psi$. Since the Proposition is proved in the same manner as Proposition \ref{p:CourantRogersKVCompat}, we omit the proof.

\begin{proposition} \label{p:FRScompat}
Let $\calG$ be an $S^1$-bundle gerbe $P\toto X$ over $M$ with connection $\gamma$, curving $B$, and resulting  3-curvature $\chi\in \Omega^3(M)$. Let $\mathsf{F}$ be as in Theorem \ref{t:butterflyCourant} and  $\mathsf{G}$  as in Theorem \ref{t:atiyah}. Then the following diagram 2-commutes:
\[
\xymatrix{
\mathbb{L}(C_\chi) \ar^{\psi}[r] & \mathbb{A}(M,\chi) \\
\mathbb{X}(\calG,\gamma) \ar[r] \ar@{-->}^{\mathsf{F}}[u] & \mathbb{X}(\mathbf{P}) \ar@{-->}_{\mathsf{G}}[u] \ultwocell\omit{}
}
\]
\end{proposition}

\begin{remark} \label{r:Gcanon}
The butterfly $\mathsf{G}: \mathbb{X}(\mathbf{P}) \dashto \mathbb{A}(M,\chi)$  in Theorem \ref{t:atiyah} depends on a choice of connection $\gamma$; however, another choice of connection would yield a 2-isomorphic butterfly. Indeed, 
 another connection must be of the form $\gamma'=\gamma+ \delta \nu$, where $\nu \in \Omega^1(X)$, and the map $(\mathbf{x},\mathbf{p},g) \mapsto (\mathbf{x},\mathbf{p},g + \iota_{\mathbf{x}} \nu)$ gives the desired 2-isomorphism. 
\end{remark}

\begin{remark} \label{r:FRS}
In \cite[Proposition 5.2.6]{fiorenza2014algebras}, the authors prove a result similar to  Propositions \ref{p:CourantRogersKVCompat} and \ref{p:FRScompat}. In \emph{op.\ cit}., the authors model $S^1$-gerbes with \v{C}ech-Deligne cocycles, which are equivalent to the data of  bundle gerbes defined in terms of \v{C}ech data (i.e., with $X=\sqcup \,U_i$ where $\{U_i \}$ is an open cover of $M$.)  The resulting Lie 2-algebras of infinitesimal symmetries (preserving the appropriate connection data) are equivalent to those in \cite{collier-PhDThesis2012}, and they establish the corresponding quasi-isomorphisms of Lie 2-algebras  and 2-commuting diagrams.
\end{remark}

\subsection{Compatibility with gauge transformations} \label{ss:courant-gauge}

In this Section, we consider the compatibility of the quasi-isomorphisms in Theorems \ref{t:butterflyCourant} and \ref{t:atiyah} with \emph{gauge transformations}, $\chi \mapsto \chi+d\tau$, where $\tau \in \Omega^2(M)$. 

We begin with a Lemma showing gauge transformations leave the isomorphism class of the Courant Lie 2-algebra invariant.

\begin{lemma} \label{p:gauge}
Let $\tau \in \Omega^2(M)$, and let $\mathsf{T}_\tau:\mathbb{L}(C_\chi) \to \mathbb{L}(C_{\chi+d\tau})$ be defined by,
\[
(\mathsf{T}_\tau)_0(u,\alpha) = (u, \alpha + \iota_u \tau), \quad (\mathsf{T}_\tau)_1 = \mathsf{id}\, .
\]
Then the chain map $(\mathsf{T}_\tau)_\bullet$ is a (strict) isomorphism of Lie 2-algebras.
\end{lemma}
\begin{proof}
This is proven in \cite{miti2022observables} for higher Courant algebroids. In this special case, it is  straightforward to verify directly. Indeed, it is obvious that $(\mathsf{T}_\tau)_\bullet$ is a chain map, and a direct calculation shows that $(\mathsf{T}_\tau)_0$ preserves brackets and the standard pairing; whence, the compatibility of Jacobiators follows.
\end{proof}

The Lie 2-algebra of connection preserving multiplicative vector fields on an $S^1$-bundle gerbe with 2-curvature $\chi$ is invariant under such gauge transformations. Indeed, by Remark \ref{r:gauge-curving}, a gauge transformation corresponds to a change in curving---in particular, the underlying bundle gerbe and connection remain the same. By Lemma \ref{l:morphs-connex-pres}, the Lie 2-algebras $\mathbb{X}(\calG,\gamma)$ resulting from the different curvings coincide.

\begin{proposition}\label{p:compatibility_gauge_transf}
Let $(\calG,\gamma)$ be an $S^1$-bundle gerbe over $M$ with connection $\gamma$. Suppose $B$ is a curving for $\gamma$ with resulting 3-curvature $\chi$, while $B'=B+\pi^*\tau$ is a curving for $\gamma$ with resulting 3-curvature $\chi+d\tau$. Let $\mathsf{F}$ and $\mathsf{F}'$ be the invertible butterflies in Theorem \ref{t:butterflyCourant} corresponding to the respective choices of curving, and let $\mathsf{T}_\tau:\mathbb{L}(C_\chi) \to \mathbb{L}(C_{\chi+d\tau})$ be as in Lemma \ref{p:gauge}. Then the diagram below 2-commutes:
\[
\xymatrix{
& \mathbb{L}(C_{\chi}) \ar[dd]^-{\mathsf{T}_\tau} \\
\mathbb{X}(\calG,\gamma) \ar@{-->}^{\mathsf{F}}[ur] \ar@{-->}_{\mathsf{F}'}[dr] \rtwocell\omit{}& \\
& \mathbb{L}(C_{\chi+d\tau})
}
\]
\end{proposition}
\begin{proof} Recall that the underlying vector space of $\mathsf{F}'$ is the same as that for $\mathsf{F}$---denote this vector space by $F$ as in \eqref{eq:butterflyF}.

Since $\mathsf{T}_\tau$ is a strict morphism of Lie 2-algebras, the underlying vector space of the butterfly of the composition $\mathsf{T}_\tau \circ \mathsf{F}$ is given by a pushout along $(\mathsf{T}_\tau)_1=\id_{K_1}$ (see \cite[Section 5.1]{noohi2013integrating}), $ (F\oplus K_1)/K_1 \cong F$, where $K_1\to K_0$ denotes the underlying 2-term complex of $\mathbb{L}(C_\chi)$.
This identification gives the desired morphism of butterflies $\mathsf{T}_\tau \circ \mathsf{F} \Rightarrow \mathsf{F}'$.
\end{proof}

Similar to Lemma \ref{p:gauge}, we see that varying $\chi$ within its cohomology class does not change the isomorphism class of the Atiyah Lie 2-algebra.

\begin{lemma} \label{p:atiyah_gauge}
Let $\tau \in \Omega^2(M)$. The identity chain map $\mathbb{A}(M,\chi) \to \mathbb{A}(M,\chi+d\tau)$ with chain homotopy  $({x}_1, {x}_1) \mapsto \iota_{{x}_2} \iota_{{x}_1} \tau$ defines an isomorphism of Lie 2-algebras.
\end{lemma}
\begin{proof}
That the above formula defines a chain homotopy follows immediately from the invariant formula for the exterior derivative.
\end{proof}

\begin{proposition}\label{p:atiyah_compatibility_gauge_transf}
Let $\calG = P\toto X$ be an $S^1$-bundle gerbe over $M$. Let  $\gamma$ be a connection for $\calG$. Suppose $B$ is a curving for $\gamma$ with resulting 3-curvature $\chi$, while $B'=B+\pi^*\tau$ is a curving for $\gamma$ with resulting 3-curvature $\chi+d\tau$. Let $\mathsf{G}$ and $\mathsf{G}'$ be the invertible butterflies in Theorem \ref{t:atiyah} corresponding to the respective choices of curving, and let $\mathsf{id}_\tau:\mathbb{A}(M,\chi) \to \mathbb{A}(M,{\chi+d\tau})$ denote the isomorphism in Lemma \ref{p:atiyah_gauge}. Then the diagram below 2-commutes:
\[
\xymatrix{
& \mathbb{A}(M,{\chi}) \ar[dd]^-{\mathsf{id}_\tau} \\
\mathbb{X}(\mathbf{P}) \ar@{-->}^{\mathsf{G}}[ur] \ar@{-->}_{\mathsf{G}'}[dr] \rtwocell\omit{}& \\
& \mathbb{A}(M,{\chi+d\tau})
}
\]
\end{proposition}
\begin{proof}
The composition $\mathsf{id}_\tau \circ \mathsf{G}$ is the butterfly $G \mathrel{\substack{{A_1}\\\oplus\\{A_0}}} (A_1 \oplus A_0)$, where $A_1 \to A_0$ denotes the underlying 2-term complex for $\mathbb{A}(M,\chi)$ (and $\mathbb{A}(M,\chi+d\tau)$). The bracket is defined component-wise; on $G$, it is given in Theorem \ref{t:atiyah}, while on $A_1\oplus A_0$ (the butterfly for $\mathsf{id}_\tau$), it is given by 
\[
[(f,{x}),(g,{z})] = (L_{x} g - L_{z} f + \iota_{{z}} \iota_{{x}} \tau, [{x},{z}]).
\]
The butterfly $\mathsf{G}'$ is given by the same vector space $G$ as for $\mathsf{G}$, but with bracket
\[
[(\mathbf{x},\mathbf{p}, g) , (\mathbf{z},\mathbf{r}, h)]
= ([(\mathbf{x},\mathbf{p}) , (\mathbf{z},\mathbf{r})] ,
L_\mathbf{x} h - L_\mathbf{z} g - \iota_\mathbf{z}  \iota_\mathbf{x}  (B+\pi^*\tau)  ).
\]
Consider the natural isomorphism $\varphi:G\to G \mathrel{\substack{{A_1}\\\oplus\\{A_0}}} (A_1 \oplus A_0)$,  sending $(\mathbf{x},\mathbf{p},g)$ to the equivalence class of $ (\mathbf{x},\mathbf{p},g;0,{x})$ (where $\mathbf{x}$ descends to ${x}$).  A direct calculation shows that $\varphi$ preserves brackets.
\end{proof}

\section{Gauge transformations and homotopy moment maps} \label{s:moment-gauge}

As an application of the results in Section \ref{s:courant}, we present a geometric argument analogous to one appearing in \cite{miti2022observables} in they symplectic case, showing that 
 Rogers' embedding of Lie 2-algebras  $\mathsf{R}:\mathbb{L}(M,\chi) \hookrightarrow \mathbb{L}(C_\chi)$ is compatible with  gauge transformations $\chi \mapsto \chi + d\tau$, where $\tau \in \Omega^2(M)$ is $G$-invariant form, after pulling back to finite dimensional Lie algebras $\mathfrak{g}=\mathrm{Lie}(G)$ along homotopy moment maps.

To that end, let $(M,\chi)$ be a pre-2-plectic manifold (i.e., where $\chi \in \Omega^3(M) $ is closed) equipped with an action of a connected Lie group $G$ that preserves ${\chi}$. Suppose the $G$-action admits an homotopy moment map $\mathsf{J}_\chi:\mathfrak{g} \to \mathbb{L}(M,\chi)$ as in \cite{callies2016homotopy}.
Given a $G$-invariant 2-form $\tau \in \Omega^2(M)^G$, we define,
\[
(\mathsf{J}_{\chi+d\tau})_0 (\xi) = (\mathsf{J}_\chi)_0(\xi) + (0,\iota_{\xi_M} \tau)
\quad \text{and} \quad J_{\chi+d\tau} (\xi\otimes {\zeta}) = J_\chi(\xi \otimes \zeta) - \iota_{\xi_M} \iota_{\zeta_M} \tau,
\]
where $\xi_M$ denotes the generating vector field corresponding to $\xi \in \mathfrak{g}$.
A straightforward computation using Cartan calculus verifies that the above defines a Lie 2-algebra morphism (i.e., a homotopy moment map) $\mathsf{J}_{\chi+d\tau} :\mathfrak{g} \to \mathbb{L}(M,\chi+d\tau)$.

\begin{proposition} \label{p:moments}
Let $(\calG,\gamma)$ be an $S^1$-bundle gerbe over $M$ with curving $B$ whose 3-curvature is $\chi$.
Let $\mathsf{E}_\chi:\mathbb{L}(M,\chi) \dashto \mathbb{X}(\calG, \gamma)$ denote the composition of the prequantization butterfly $ \mathbb{L}(M,\chi) \dashto \mathbb{X}(\calG,\gamma,B)$ with inclusion into $\mathbb{X}(\calG,\gamma)$.
Then the following diagram 2-commutes:
\[
\xymatrix@C=4em{
\mathfrak{g} \ar[d]_-{J_{\chi+d\tau}} \ar[r]^-{J_\chi} \drtwocell\omit{} & \mathbb{L}(M,\chi) \ar@{-->}[d]^-{\mathsf{E}_\chi} \\
\mathbb{L}(M,\chi+d\tau) \ar@{-->}[r]_-{\mathsf{E}_{\chi+d\tau}} & \mathbb{X}(\calG,\gamma)
}
\]
\end{proposition}
\begin{proof}
The composition,
\[
\xymatrix{
\mathfrak{g} \ar[r]^-{J_\chi} & \mathbb{L}(M,\chi) \ar@{-->}[r]^-{\mathsf{E}_\chi} & \mathbb{X}(\calG,\gamma)}
\]
is a butterfly with underlying vector space $\mathfrak{g} \oplus_{L_0} E$, the fibre product of the butterfly structure map $E \to L_0$ with the map $(\mathsf{J}_\chi)_0:\mathfrak{g} \to L_0$. Similarly, the composition
\[
\xymatrix{
\mathfrak{g} \ar[r]^-{J_{\chi+d\tau}} & \mathbb{L}(M,\chi+d\tau) \ar@{-->}[r]^-{\mathsf{E}_{\chi+d\tau}} & \mathbb{X}(\calG,\gamma).
}
\]
is a butterfly with underlying vector space $\mathfrak{g} \oplus_{L_0} E$, the fibre product of $E \to L_0$ with $(\mathsf{J}_{\chi + d\tau})_0:\mathfrak{g} \to L_0$. The identity map on $\mathfrak{g} \oplus_{L_0} E$ gives the desired morphism of butterflies.
\end{proof}

Thus, in the above setting, if $(M,\chi)$ admits a prequantization bundle gerbe with connection $(\calG,\gamma)$ (i.e., whose 3-curvature is $\chi)$, then by Propositions \ref{p:CourantRogersKVCompat}, \ref{p:compatibility_gauge_transf}, and \ref{p:moments}, the following diagram  2-commutes:
\begin{equation} \label{eq:MZ}
\begin{gathered}
\xymatrix{
& \mathbb{L}(M,\chi) \ar[rr]^-{\mathsf{R}} \ar@{-->}[rd]_{\mathsf{E}_{\chi}} &  & \mathbb{L}(C_\chi) \ar[dd]^{\mathsf{T}_\tau} \\
\mathfrak{g} \rrtwocell\omit{}  \ar[ru] ^-{J_\chi} \ar[rd]_-{J_{\chi+d\tau}} & & \mathbb{X}(\calG,\gamma) \utwocell\omit{} \rtwocell\omit{} \ar@{-->}[ru]_{\mathsf{F}} \ar@{-->}[rd]^{\mathsf{F}'} & \\
& \mathbb{L}(M,\chi+d\tau) \ar[rr]_{\mathsf{R}} \ar@{-->}[ru]
^{\mathsf{E}_{\chi+d\tau}} & \utwocell\omit{} & \mathbb{L}(C_{\chi+d\tau}) \\
}
\end{gathered}
\end{equation}
Note (\emph{cf}.\ Remark 1.6 in \cite{miti2022observables}) that one can check directly  that the two compositions of Lie 2-algebra morphisms along the outer edge of diagram \eqref{eq:MZ} agree  (without requirement that $\chi$ be integral). 
The 2-commutativity of diagram  \eqref{eq:MZ} gives a geometric interpretation to that observation, showing the two compositions are 2-isomorphic, which is analogous to the geometric argument appearing in 
\cite[Section 1]{miti2022observables} in the symplectic case.

\bibliographystyle{plain}

\end{document}